\documentclass{article}
\usepackage{amssymb}
\usepackage{amsmath}
\usepackage{amsthm}
\usepackage{url}
\usepackage{enumitem}
\usepackage[left=3cm,right=3cm]{geometry}
\title{Fixed Points for Banach and Kannan Contractions in Modular Spaces with a Graph\\[0.3cm]}
\author{{Aris Aghanians$^1$\,\,\,and\,\,\,Kourosh Nourouzi$^2$\thanks{Corresponding
author} \thanks{e-mail: nourouzi@kntu.ac.ir; fax: +98 21
22853650}}\\[0.4cm]
{\em $^{1,2}$Department of Mathematics, K. N. Toosi University of Technology,}\\
{\em P.O. Box 16315-1618, Tehran, Iran}}
\theoremstyle{definition}
\newtheorem{defn}{Definition}
\newtheorem{exm}{Example}
\theoremstyle{remark}
\newtheorem{rem}{Remark}
\theoremstyle{plain}
\newtheorem{cor}{Corollary}
\newtheorem{thm}{Theorem}
\newtheorem{prop}{Proposition}

\DeclareMathOperator{\fix}{Fix}
\begin{document}
\maketitle
\begin{abstract}
In this paper, we discuss the existence and uniqueness of fixed
points for Banach and Kannan $\widetilde G$-$\rho$-contractions defined on modular
spaces endowed with a graph without using the $\Delta_2$-condition or the Fatou property. 
\end{abstract}
\def\thefootnote{ \ }
\footnotetext{{\em}$2010$ Mathematics Subject Classification.
47H10, 46A80, 05C40.\par {\bf Keywords:} Complete modular space; Fixed point; Banach contraction;
Kannan contraction.}

\section{Introduction and Preliminaries}
To control the pathological behavior of a modular in modular spaces some  conditions such as
the $\Delta_2$-condition and the Fatou property are usually assumed (see, e.g., \cite{ait,nou091,kha, nou12,nou092,raz}.
For instance, in \cite{ait}, Banach's fixed point theorem is given in modular
spaces having both the $\Delta_2$-condition and the Fatou
property. Khamsi \cite{kha}, also presented some fixed point
theorems for quasi-contractions in modular spaces satisfying only the Fatou property.\par
In 2008 Jachymski \cite{jac} established Banach fixed point theorem in metric spaces with a graph and his idea followed by the authors in uniform spaces (see, e.g., \cite{knpan,knfpt}).\par
In this paper motivated by the ideas given in \cite{ait,jac}, we aim to discuss the fixed points of Banach and Kannan
contractions in modular spaces endowed with a graph without $\Delta_2$-condition and  Fatou property. We also clarify the independence of these contractions.
\par
We first commence some basic concepts about modular
spaces as formulated by Musielak and Orlicz \cite{mus}. For more
details, the reader is referred to \cite{mus0}.

\begin{defn}
A real-valued function $\rho$ defined on a real vector space $X$
is called a modular on $X$ if it satisfies the following conditions:
\begin{enumerate}[label={\bf M\arabic*)}]
\item $\rho(x)\geq0$ for all $x\in X$;
\item $\rho(x)=0$ if and only if $x=0$;
\item $\rho(x)=\rho(-x)$ for all $x\in X$;
\item $\rho(ax+by)\leq\rho(x)+\rho(y)$ for all $x,y\in X$ and
all $a,b\geq0$ with $a+b=1$.
\end{enumerate}
If $\rho$ satisfies (M1)-(M4), then the pair $(X,\rho)$, shortly denoted by $X$, is called a modular space.
\end{defn}

The modular $\rho$ is called convex if Condition (M4) is strengthened by replacing with
\begin{enumerate}[label={\bf M\arabic*$\boldsymbol{'}$)},leftmargin=9.5mm]
\setcounter{enumi}{3}
\item $\rho(ax+by)\leq a\rho(x)+b\rho(y)$ for all $x,y\in X$ and all $a,b\geq0$ with $a+b=1$.
\end{enumerate}\par
It is easy to obtain the following two immediate consequences of Condition (M4) which we need in the sequel:
\begin{itemize}
\item If $a$ and $b$ are real numbers with $|a|\leq|b|$, then $\rho(ax)\leq\rho(bx)$ for all $x\in X$;
\item If $a_1,\ldots,a_n$ are nonnegative numbers with $\sum_{i=1}^na_i=1$, then
$$\rho\Big(\sum_{i=1}^na_ix_i\Big)\leq\sum_{i=1}^n\rho(x_i)\qquad(x_1,\ldots,x_n\in X).$$
\end{itemize}

\begin{defn}\label{space}
Let $(X,\rho)$ be a modular space.
\begin{enumerate}
\item A sequence $\{x_n\}$ in $X$ is said to be
$\rho$-convergent to a point $x\in X$, denoted by
$x_n\stackrel{\rho}\longrightarrow x$, if $\rho(x_n-x)\rightarrow 0$
as $n\rightarrow\infty$. \item A sequence $\{x_n\}$ in $X$
is said to be $\rho$-Cauchy if $\rho(x_m-x_n)\rightarrow0$ as
$m,n\rightarrow\infty$. \item The modular space $X$ is called
$\rho$-complete if each $\rho$-Cauchy sequence in $X$ is
$\rho$-convergent to a point of $X$.
\item The modular $\rho$ is
said to satisfy the $\Delta_2$-condition if $2x_n\stackrel{\rho}\longrightarrow0$ as $n\rightarrow\infty$ whenever $x_n\stackrel{\rho}\longrightarrow0$ as $n\rightarrow\infty$.
\item The modular $\rho$ is said to have the Fatou property if
$$\rho(x-y)\leq\liminf\limits_{n\rightarrow\infty}\rho(x_n-y_n)$$
whenever
$$x_n\stackrel{\rho}\longrightarrow x\quad\text{and}\quad
y_n\stackrel{\rho}\longrightarrow y\quad\text{as}\quad n\rightarrow\infty.$$
\end{enumerate}
\end{defn}

Conditions (M2) and (M4) ensure that each sequence in a modular space can be $\rho$-convergent to at most one point. In other words, the limit of a $\rho$-convergent sequence in a modular space is unique.\par
We next review some notions in graph theory. All of them can be found in, e.g., \cite{bon}.\par
Let $X$ be a modular space. Consider a directed graph $G$ with $V(G)=X$ and $E(G)\supseteq\{(x,x):x\in X\}$, i.e., $E(G)$ contains all loops. Suppose further that $G$ has no parallel edges. With these assumptions, we may denote $G$ by the pair $(V(G),E(G))$. In this way, the modular space $X$ is endowed with the graph $G$. The notation $\widetilde G$ is used to denote the undirected graph obtained from $G$ by deleting the directions of the edges of $G$. Thus,
$$V(G)=X\quad\text\quad E(G)=\big\{(x,y)\in X\times X:(x,y)\in E(G)\ \vee\ (y,x)\in E(G)\big\}.$$\par
By a path in $G$ from a vertex $x$ to a vertex $y$, it is meant a finite sequence $(x_s)_{s=0}^N$ of vertices of $G$ such that $x_0=x$, $x_N=y$, and $(x_{s-1},x_s)\in E(G)$ for $s=1,\ldots,N$. A graph $G$ is called weakly connected if there exists a path in $\widetilde G$ between each two vertices of $G$, i.e., there exists an undirected path in $G$ between its each two vertices.

\section{Main Results}
Let $X$ be a modular space endowed with a graph $G$ and $f:X\to X$ be any mapping. The set of all fixed points for $f$ is denoted by $\fix(f)$, and by $C_f$, we mean the set of all elements $x$ of $X$ such that $(f^nx,f^mx)\in E(\widetilde G)$ for $m,n=0,1,\ldots\,$.\par
We begin with introducing Banach and Kannan $G$-$\rho$-contractions. 

\begin{defn}\label{bankan}
Let $X$ be a modular space with a graph $G$ and $f:X\to X$ be a mapping. We call $f$ a Banach $G$-$\rho$-contraction if
\begin{enumerate}[label={\bf B\arabic*)}]
\item $f$ preserves the edges of $G$, i.e., $(x,y)\in E(G)$ implies $(fx,fy)\in E(G)$ for all $x,y\in X$;
\item there exist positive numbers $k$, $a$ and $b$ with $k<1$ and $a<b$ such that
$$\rho\big(b(fx-fy)\big)\leq k\rho\big(a(x-y)\big)$$
for all $x,y\in X$ with $(x,y)\in E(G)$.
\end{enumerate}
The numbers $k$, $a$ and $b$ are called the constants of $f$. And we call $f$ a Kannan $G$-$\rho$-contraction (see \cite{kan, rho2}) if
\begin{enumerate}[label={\bf K\arabic*)}]
\item $f$ preserves the edges of $G$;
\item there exist positive numbers $k,l,a_1,a_2$ and $b$  with $k+l<1$, $a_1\leq\frac{b}{2}$ and
$a_2\leq b$ such that
$$\rho\big(b(fx-fy)\big)\leq k\rho\big(a_1(fx-x)\big)+l\rho\big(a_2(fy-y)\big)$$
for all $x,y\in X$ with $(x,y)\in E(G)$.
\end{enumerate}
The numbers $k$, $l$, $a_1$, $a_2$ and $b$ are called the constants of $f$.
\end{defn}

It might be valuable if we discuss these contractions a little. Our first proposition follows immediately from Condition (M3) and Definition \ref{bankan}.

\begin{prop}
Let $X$ be a modular space with a graph $G$. If a mapping from $X$ into itself satisfies {\rm(B1)}
(respectively, {\rm(B2)}) for $G$, then it satisfies
{\rm(B1)} (respectively, {\rm(B2)}) for $\widetilde G$.
In particular, a Banach $G$-$\rho$-contraction is also a
Banach $\widetilde G$-$\rho$-contraction. Similar statements are true for Kannan $G$-$\rho$-contractions provided that $a_2\leq\frac b2$.
\end{prop}

We also have the following remark about Kannan $G$-$\rho$-contractions.

\begin{rem}
For a Kannan $\widetilde G$-$\rho$-contraction $f:X\rightarrow X$, we can interchange the roles of $x$ and $y$ in (K2) since $E(\widetilde G)$ is symmetric. Having done this, we find
\begin{eqnarray*}
\rho\big(b(fx-fy)\big)&=&\rho\big(b(fy-fx)\big)\cr\\[-4mm]
&\leq&k\rho\big(a_1(fy-y)\big)+l\rho\big(a_2(fx-x)\big)\cr\\[-4mm]
&=&l\rho\big(a_2(fx-x)\big)+k\rho\big(a_1(fy-y)\big).
\end{eqnarray*}
Therefore, no matter $a_1\leq\frac b2$ or $a_2\leq\frac b2$ whenever we are faced with Kannan $\widetilde G$-$\rho$-contractions. Nevertheless, both $a_1$ and $a_2$ must be no more than $b$.
\end{rem}

We now give some examples.

\begin{exm}
Let $X$ be a modular space with any arbitrary graph $G$. Since $E(G)$ contains all loops, each constant mapping $f:X\to X$ is both a Banach and a Kannan $G$-$\rho$-contraction. In fact, $E(G)$ should contain all loops if we want any constant mapping to be either a Banach or a Kannan $G$-$\rho$-contraction.
\end{exm}

\begin{exm}
Let $X$ be a modular space and $G_0$ be the complete graph $(X,X\times X)$. Then Banach (Kannan) $G_0$-$\rho$-contractions are precisely the Banach (Kannan) contractions in modular spaces.
\end{exm}

\begin{exm}
Let $\preceq$ be a partial order on a modular space $X$ and consider a poset graph $G_1$ by $V(G_1)=X$ and $E(G_1)=\big\{(x,y)\in X\times X:x\preceq y\big\}$. Then Banach $G_1$-$\rho$-contractions are precisely the nondecreasing ordered $\rho$-contractions. A similar statement is true for Kannan $G_1$-$\rho$-contractions.
\end{exm}

Finally, we show that Banach and Kannan $G$-$\rho$-contractions are independent of each other. More precisely, we construct two mappings on $\Bbb{R}$ such that one of them satisfies (B2) but not (K2), and the other, (K2) but not (B2) for the complete graph $G_0$.

\begin{exm}
Let $\rho$ be the usual Euclidean norm on $\Bbb{R}$,
i.e., $\rho(x)=|x|$ for all $x\in\Bbb{R}$. Define a mapping $f:\Bbb{R}\rightarrow\Bbb{R}$ by
$fx=\frac{x}{3}$ for all $x\in\Bbb{R}$. Then $f$ is a Banach $G_0$-$\rho$-contraction with the constants
$k=\frac{2}{3}$, $a=\frac{1}{2}$ and $b=1$. Indeed, given any $x,y\in\Bbb{R}$, we have
$$\rho\big(b(fx-fy)\big)=\frac13|x-y|=k\rho\big(a(x-y)\big).$$
On the other hand, if $k$, $l$, $a_1$, $a_2$ and $b$ are any arbitrary positive numbers satisfying $k+l<1$, $a_1\leq \frac{b}{2}$ and
$a_2\leq b$, then for $y=0$ and any $x\ne0$ we see that
$$\rho\big(b(fx-f0)\big)=\frac{b|x|}{3}>\frac{2a_1k|x|}3=k\rho\big(a_1(fx-x)\big)+l\rho\big(a_2(f0-0)\big).$$
Therefore, (K2) fails to hold and $f$ is not a Kannan $G_0$-$\rho$-contraction.
\end{exm}

\begin{exm}
It is easy to verify that the function $\rho(x)=x^2$ defines a modular on $\Bbb{R}$ and $(\Bbb{R},\rho)$ is $\rho$-complete because $(\Bbb{R},|\cdot|)$ is a Banach space. Now, consider a mapping $f:\Bbb{R}\rightarrow\Bbb{R}$ defined by $fx=\frac{1}{2}$ if $x\neq1$ and $f1=\frac{1}{10}$. Then $f$ is a $G_0$-$\rho$-Kannan
contraction with the constants $k=\frac{64}{81}$,
$l=\frac{16}{81}$, $a_1=\frac{1}{2}$ and $a_2=b=1$. Indeed, given any
$x,y\in\Bbb{R}$, we have the following three possible cases:
\begin{enumerate}[label={Case \arabic*:},leftmargin=1.7cm]
\item If $x=y=1$ or $x,y\neq1$, then (K2) holds trivially
since $fx=fy$;
\item If $x=1$ and $y\neq1$, then
$$\rho\big(b(fx-fy)\big)=\frac{4}{25}\leq\frac{4}{25}+\frac{16}{81}\Big(\frac{1}{2}-y\Big)^2=k\rho\big(a_1(fx-x)\big)+l\rho\big(a_2(fy-y)\big);$$
\item Finally, if $x\neq1$ and $y=1$, then
$$\rho\big(b(fx-fy)\big)=\frac{4}{25}\leq\frac{16}{81}\Big(\frac{1}{2}-x\Big)^2+\frac{4}{25}=k\rho\big(a_1(fx-x)\big)+l\rho\big(a_2(fy-y)\big).$$
\end{enumerate}
Note that $k+l=\frac{80}{81}<1$, $a_1\leq\frac{b}{2}$ and $a_2\leq b$.  But $f$ is not a
Banach $G_0$-$\rho$-contraction; for if $k$, $a$ and $b$ are any arbitrary positive numbers satisfying
$k<1$ and $a<b$, then putting $x=1$ and $y=\frac{3}{5}$ yields
$$\rho\big(b(fx-fy)\big)=\frac{4b^2}{25}>\frac{4a^2k}{25}=k\rho\big(a(x-y)\big).$$
\end{exm}

Now we are going to prove our fixed point results. The first one is about the existence and uniqueness of fixed points for Banach $\widetilde G$-$\rho$-contractions.

\begin{thm}\label{banach}
Let $X$ be a $\rho$-complete modular space endowed with a graph $G$ and the triple $(X,\rho,G)$ have the following property:
\begin{itemize}[label={$(\ast)$}]
\item If $\{x_n\}$ is a sequence in $X$ such that $\beta x_n\stackrel{\rho}\longrightarrow\beta x$ for some $\beta>0$ and $(x_n,x_{n+1})\in E(\widetilde G)$ for all $n\geq1$, then there exists a subsequence $\{x_{n_i}\}$ of $\{x_n\}$ such that $(x_{n_i},x)\in E(\widetilde G)$ for all $i\geq1$.
\end{itemize}
Then a Banach $\widetilde G$-$\rho$-contraction $f:X\to X$ has a fixed point if and only if $C_f\ne\emptyset$. Moreover, this fixed point is unique if $G$ is weakly connected.
\end{thm}

\begin{proof}
$(\Rightarrow)$ It is trivial since $\fix(f)\subseteq C_f$.\par
$(\Leftarrow)$ Let $k$, $a$ and $b$ be the constants of $f$ and let $\alpha>1$ be the exponential conjugate of $\frac{b}{a}$,
i.e., $\frac{a}{b}+\frac{1}{\alpha}=1$. Choose an $x\in C_f$ and keep it fixed. We are going to show that the sequence $\{bf^nx\}$ is $\rho$-Cauchy in $X$. To this end, note first if $n$ is a positive integer, then by (B2) we have
\begin{eqnarray*}
\rho\big(a(f^nx-x)\big)&=&\rho\big(a(f^nx-fx)+a(fx-x)\big)\cr\\[-3mm]
&=&\rho\Big(\frac{a}{b}b\big(f^nx-fx\big)+\frac{1}{\alpha}\alpha a\big(fx-x\big)\Big)\cr\\[-3mm]
&\leq&\rho\big(b(f^nx-fx)\big)+\rho\big(\alpha a(fx-x)\big)\cr\\[-3mm]
&\leq&k\rho\big(a(f^{n-1}x-x)\big)+r,
\end{eqnarray*}
where $r=\rho(\alpha a(fx-x))$. Hence using the mathematical induction, we get
\begin{eqnarray*}
\rho\big(a(f^nx-x)\big)&\leq&k\rho\big(a(f^{n-1}x-x)\big)+r\cr\\[-4mm]
&\leq&k\Big[k\rho\big(a(f^{n-2}x-x)\big)+r\Big]+r\cr\\[-4mm]
&=&k^2\rho\big(a(f^{n-2}x-x)\big)+kr+r\cr\\[-4mm]
&\vdots&\cr\\[-4mm]
&\leq&k^{n-1}\rho\big(a(fx-x)\big)+k^{n-2}r+\cdots+r
\end{eqnarray*}
for all $n\geq1$. Since $\alpha>1$, it follows that $\rho(a(fx-x))\leq\rho(\alpha a(fx-x))=r$ and therefore,
\begin{equation}\label{eqn2}
\rho\big(a(f^nx-x)\big)\leq k^{n-1}r+\cdots+r=\frac{(1-k^n)r}{1-k}\leq\frac r{1-k}\qquad n=1,2,\ldots\,.
\end{equation}
Now using (B2) once more, we find
\begin{eqnarray}\label{eqn1}
\rho\big(b(f^mx-f^nx)\big)&\leq&k\rho\big(a(f^{m-1}x-f^{n-1}x)\big)\cr\notag\\[-4mm]
&\leq&k\rho\big(b(f^{m-1}x-f^{n-1}x)\big)\cr\notag\\[-4mm]
&\vdots&\cr\notag\\[-4mm]
&\leq&k^n\rho\big(a(f^{m-n}x-x)\big)
\end{eqnarray}
for all $m$ and $n$ with $m>n\geq1$. Consequently, by combining (\ref{eqn2}) and (\ref{eqn1}), it is seen that for
all $m>n\geq1$ we have
$$\rho\big(b(f^mx-f^nx)\big)\leq k^n\rho\big(a(f^{m-n}x-x)\big)\leq\frac{k^nr}{1-k}.$$
Therefore, $\rho(b(f^mx-f^nx))\rightarrow0$ as $m,n\rightarrow\infty$, and so $\{bf^nx\}$ is a $\rho$-Cauchy sequence in
$X$ and because $X$ is $\rho$-complete, it is $\rho$-convergent. On the other hand, $X$ is a real vector space and $b>0$. Thus, there exists an $x^*\in X$ such that $bf^nx\stackrel{\rho}\longrightarrow bx^*$.\par
We next show that $x^*$ is a fixed point for $f$. Since $x\in C_f$, it follows that $(f^nx,f^{n+1}x)\in E(\widetilde G)$ for all $n\geq0$, and so by Property $(\ast)$, there exists a strictly increasing sequence $\{n_i\}$ of positive integers such that $(f^{n_i}x,x^*)\in E(\widetilde G)$ for all $i\geq1$. Hence using (B2) we get
\begin{eqnarray*}
\rho\Big(\frac{b}{2}\big(fx^*-x^*\big)\Big)&=&\rho\Big(\frac{b}{2}\big(fx^*-f^{n_i+1}x\big)+\frac{b}{2}\big(f^{n_i+1}x-x^*\big)\Big)\cr\notag\\[-3mm]
&\leq&\rho\big(b(fx^*-f^{n_i+1}x)\big)+\rho\big(b(f^{n_i+1}x-x^*)\big)\cr\notag\\[-3mm]
&=&\rho\big(b(f^{n_i+1}x-fx^*)\big)+\rho\big(b(f^{n_i+1}x-x^*)\big)\cr\notag\\[-3mm]
&\leq&k\rho\big(a(f^{n_i}x-x^*)\big)+\rho\big(b(f^{n_i+1}x-x^*)\big)\cr\notag\\[-3mm]
&\leq&k\rho\big(b(f^{n_i}x-x^*)\big)+\rho\big(b(f^{n_i+1}x-x^*)\big)\rightarrow0
\end{eqnarray*}
as $i\rightarrow\infty$. So $\rho(\frac{b}{2}(fx^*-x^*))=0$, and since $b>0$, it follows that $fx^*-x^*=0$ or equivalently, $fx^*=x^*$, i.e., $x^*$ is a fixed point for $f$.\par
Finally, to prove the uniqueness of the fixed point, suppose that $G$ is weakly connected and $y^*\in X$ is a fixed point for $f$. Then there exists a path $(x_s)_{s=0}^N$ in $\widetilde G$ from $x^*$ to $y^*$, i.e., $x_0=x^*$, $x_N=y^*$, and $(x_{s-1},x_s)\in E(\widetilde G)$ for $s=1,\ldots,N$. Thus, by (B1), we have
$$(f^nx_{s-1},f^nx_s)\in E(\widetilde G)\qquad(n\geq0\ \text{and}\  s=1,\ldots,N).$$
And using (B2) and the mathematical induction we get
\begin{eqnarray*}
\rho\Big(\frac bN\big(x^*-y^*\big)\Big)&=&\rho\Big(\frac bN\big(x^*-f^nx_1\big)+\cdots+\frac bN\big(f^nx_{N-1}-y^*\big)\Big)\cr\\[-3mm]
&\leq&\rho\big(b(x^*-f^nx_1)\big)+\cdots+\big(b(f^nx_{N-1}-y^*)\big)\cr\\[-3mm]
&=&\sum_{s=1}^N\rho\big(b(f^nx_{s-1}-f^nx_s)\big)\cr\\
&\leq&k\sum_{s=1}^N\rho\big(a(f^{n-1}x_{s-1}-f^{n-1}x_s)\big)\cr\\
&\leq&k\sum_{s=1}^N\rho\big(b(f^{n-1}x_{s-1}-f^{n-1}x_s)\big)\cr\\
&\vdots&\cr\\[-3mm]
&\leq&k^n\sum_{s=1}^N\rho\big(b(x_{s-1}-x_s)\big)\rightarrow0
\end{eqnarray*}
as $n\rightarrow\infty$. So $\frac bN(x^*-y^*)=0$, and since $b>0$, it follows that $x^*=y^*$. Consequently, the fixed point of $f$ is unique.
\end{proof}

Setting $G=G_0$ and $G=G_1$, we get the following consequences of Theorem \ref{banach} in modular and partially ordered modular spaces, respectively.

\begin{cor}
Let $X$ be a $\rho$-complete modular space and a mapping $f:X\to X$ satisfies
$$\rho\big(b(fx-fy)\big)\leq k\rho\big(a(x-y)\big)\qquad(x,y\in X),$$
where $0<k<1$ and $0<a<b$. Then $f$ has a unique fixed point $x^*\in X$ and $bf^nx\stackrel{\rho}\longrightarrow bx^*$ for all $x\in X$.
\end{cor}

\begin{cor}
Let $\preceq$ be a partial order on a $\rho$-complete modular space $X$ such that the triple $(X,\rho,\preceq)$ has the following property:
\begin{itemize}[label={$(\ast\ast)$}]
\item If $\{x_n\}$ is a nondecreasing sequence in $X$ such that $\beta x_n\stackrel{\rho}\longrightarrow\beta x$ for some $\beta>0$, then there exists a subsequence $\{x_{n_i}\}$ of $\{x_n\}$ such that $x_{n_i}\preceq x$ for all $i\geq1$.
\end{itemize}
Assume that a nondecreasing mapping $f:X\to X$ satisfies
$$\rho\big(b(fx-fy)\big)\leq k\rho\big(a(x-y)\big)\qquad(x,y\in X\ \text{and}\ x\preceq y),$$
where $0<k<1$ and $0<a<b$. Then $f$ has a fixed point if and only if there exists an $x\in X$ such that $T^nx$ is comparable to $T^mx$ for all $m,n\geq0$. Moreover, this fixed point is unique if the following condition holds:
\begin{itemize}[label={}]
\item For all $x,y\in X$, there exists a finite sequence $(x_s)_{s=0}^N$ in $X$ with comparable successive terms such that $x_0=x$ and $x_N=y$.
\end{itemize}
\end{cor}


\begin{cor}
Let $(X,\rho)$ be a $\rho$-complete modular space endowed with a graph $G$, where $\rho$ is a convex modular, and the triple $(X,\rho,G)$ have Property $(\ast)$. Assume that $f:X\to X$ is a mapping which preserves the edges of $\widetilde G$ and satisfies
$$\rho\big(b(fx-fy)\big)\leq k\rho\big(a(x-y)\big)\qquad\big(x,y\in X\ \text{and}\ (x,y)\in E(\widetilde G)\big),$$
where $k$, $a$ and $b$ are positive numbers with $b>\max\{a,ak\}$. Then $f$ has a fixed point if and only if $C_f\ne\emptyset$. Moreover, this fixed point is unique if $G$ is weakly connected.
\end{cor}

\begin{proof}
Set $c=\max\{a,ak\}$ and choose any $a_0\in(c,b)$. Then by the hypothesis and convexity of $\rho$, we have
\begin{eqnarray*}
\rho\big(b(fx-fy)\big)&\leq&k(\rho\big(a(x-y)\big)\cr\\[-2mm]
&=&k\rho\Big(\frac a{a_0}a_0\big(x-y\big)+\big(1-\frac a{a_0}\big)0\Big)\cr\\[-2mm]
&\leq&\frac{ak}{a_0}\rho\big(a_0(x-y)\big)
\end{eqnarray*}
for all $x,y\in X$ with $(x,y)\in E(\widetilde G)$. Since $a_0<b$, and $\frac{ak}{a_0}<1$, it follows that $f$ satisfies (B2) for the graph $\widetilde G$ with the constants $k$ and $a$ replaced with $\frac{ak}{a_0}$ and $a_0$, respectively, and $b$ kept fixed. Since $f$ preserves the edges of $\widetilde G$, it follows that $f$ is a Banach $\widetilde G$-$\rho$-contraction and the results are concluded immediately from Theorem \ref{banach}.
\end{proof}

Our next result is about the existence and uniqueness of fixed points for Kannan $\widetilde G$-$\rho$-contractions.

\begin{thm}\label{kannan}
Let $X$ be a $\rho$-complete modular space endowed with a graph $G$ and the triple $(X,\rho,G)$ have Property $(\ast)$. Then a Kannan $\widetilde G$-$\rho$-contraction $f:X\to X$ has a fixed point if and only if $C_f\ne\emptyset$. Moreover, this fixed point is unique if $k<\frac12$ and $X$ satisfies the following condition:
\begin{itemize}[label={$(\star)$}]
\item For all $x,y\in X$, there exists a $z\in X$ such that $(x,z),(y,z)\in E(\widetilde G)$.
\end{itemize}
\end{thm}
\begin{proof}
$(\Rightarrow)$ It is trivial since $\fix(f)\subseteq C_f$.\par
$(\Leftarrow)$ Let $k$, $l$, $a_1$, $a_2$ and $b$ be the constants of $f$. Choose an $x\in C_f$ and keep it fixed. We are going to show that the sequence $\{bf^nx\}$ is $\rho$-Cauchy in $X$. Given any integer $n\geq2$, by (K2) we have
\begin{eqnarray*}
\rho\big(b(f^nx-f^{n-1}x)\big)&\leq&k\rho\big(a_1(f^nx-f^{n-1}x)\big)+l\rho\big(a_2(f^{n-1}x-f^{n-2}x)\big)\cr\\[-4mm]
&\leq&k\rho\big(b(f^nx-f^{n-1}x)\big)+l\rho\big(b(f^{n-1}x-f^{n-2}x)\big),
\end{eqnarray*}
which yields
$$\rho\big(b(f^nx-f^{n-1}x)\big)\leq\delta\rho\big(b(f^{n-1}x-f^{n-2}x)\big),$$
where $\delta=\frac{l}{1-k}\in(0,1)$. Hence using the mathematical induction, we get
$$\rho\big(b(f^nx-f^{n-1}x)\big)\leq\delta^n\rho\big(b(fx-x)\big)\qquad n=1,2,\ldots\,.$$
Now using (K2) once more, we find
\begin{eqnarray*}
\rho\big(b(f^mx-f^nx)\big)&\leq&k\rho\big(a_1(f^mx-f^{m-1}x)\big)+l\rho\big(a_2(f^nx-f^{n-1}x)\big)\cr\\[-4mm]
&\leq&k\rho\big(b(f^mx-f^{m-1}x)\big)+l\rho\big(b(f^nx-f^{n-1}x)\big)\cr\\[-4mm]
&\leq& k\delta^m\rho\big(b(fx-x)\big)+l\delta^n\rho\big(b(fx-x)\big)
\end{eqnarray*}
for all $m,n\geq1$. Therefore, $\rho(b(f^mx-f^nx))\rightarrow0$ as $m,n\rightarrow\infty$, and so $\{bf^nx\}$ is a $\rho$-Cauchy sequence in $X$ and because $X$ is $\rho$-complete, it is $\rho$-convergent. Thus, there
exists an $x^*\in X$ such that $bf^nx\stackrel{\rho}\longrightarrow bx^*$.\par
We next show that $x^*$ is a fixed point for $f$. Since $x\in C_f$, it follows that $(f^nx,f^{n+1}x)\in E(\widetilde G)$ for all $n\geq0$, and so by Property $(\ast)$, there exists a strictly increasing sequence $\{n_i\}$ of positive integers such that $(f^{n_i}x,x^*)\in E(\widetilde G)$ for all $i\geq1$. Hence using (K2), we get
\begin{eqnarray*}
\rho\Big(\frac{b}{2}\big(fx^*-x^*\big)\Big)&=&\rho\Big(\frac{b}{2}\big(fx^*-f^{n_i+1}x\big)+\frac{b}{2}\big(f^{n_i+1}x-x^*\big)\Big)\cr\\[-3mm]
&\leq&\rho\big(b(fx^*-f^{n_i+1}x)\big)+\rho\big(b(f^{n_i+1}x-x^*)\big)\cr\\[-3mm]
&\leq&\Big[k\rho\big(a_1(fx^*-x^*)\big)+l\rho\big(a_2(f^{n_i+1}x-f^{n_i}x)\big)\Big]+\rho\big(b(f^{n_i+1}x-x^*)\big)\cr\\[-3mm]
&\leq&k\rho\Big(\frac{b}{2}\big(fx^*-x^*\big)\Big)+l\rho\big(b(f^{n_i+1}x-f^{n_i}x)\big)+\rho\big(b(f^{n_i+1}x-x^*\big)\Big)
\end{eqnarray*}
for all $k\geq1$. Hence
$$\rho\Big(\frac{b}{2}\big(fx^*-x^*\big)\Big)\leq\delta\rho\big(b(f^{n_i+1}x-f^{n_i}x)\big)+\frac{1}{1-k}\rho\big(b(f^{n_i+1}x-x^*)\big)\rightarrow0$$
as $i\rightarrow\infty$. So $\rho(\frac{b}{2}(fx^*-x^*))=0$, and since $b>0$, it follows that $fx^*-x^*=0$ or equivalently, $fx^*=x^*$, i.e., $x^*$ is a fixed point for $f$.\par
Finally, to prove the uniqueness of the fixed point, suppose that Condition $(\star)$ holds and $y^*\in X$ is a fixed point for $f$. We consider the following two cases:\par\vspace{2.5mm}
{\bf Case 1: $\boldsymbol{(x^*,y^*)}$ is an edge of $\boldsymbol{\widetilde G}$.}\par
In this case, using (K2), we find
$$\rho\big(b(x^*-y^*)\big)=\rho\big(b(fx^*-fy^*)\big)\leq k\rho\big(a_1(fx^*-x^*)\big)+l\rho\big(b(fy^*-y^*)\big)=0.$$
Therefore, $\rho(b(x^*-y^*))=0$, and so $x^*=y^*$ because $b>0$.\par\vspace{2.5mm}
{\bf Case 2: $\boldsymbol{(x^*,y^*)}$ is not an edge of $\boldsymbol{\widetilde G}$.}\par
In this case, by Condition $(\star)$, there exists a $z\in X$ such that both $(x^*,z)$ and $(y^*,z)$ are edges of $\widetilde G$. So by (K1), we have $(x^*,f^nz),(y^*,f^nz)\in E(\widetilde G)$ for all $n\geq0$ since $x^*$ is a fixed point for $f$. Therefore, by (K2) we find
\begin{eqnarray*}
\rho\big(b(f^nz-x^*)\big)&=&\rho\big(b(f^nz-f^nx^*)\big)\cr\\[-2mm]
&\leq&k\rho\big(a_1(f^nz-f^{n-1}z)\big)+l\rho\big(a_2(f^nx^*-f^{n-1}x^*)\big)\cr\\[-2mm]
&\leq&k\rho\Big(\frac b2\big(f^nz-f^{n-1}z\big)\Big)\cr\\[-2mm]
&=&k\rho\Big(\frac b2\big(f^nz-f^nx^*\big)+\frac b2\big(f^{n-1}x^*-f^{n-1}z\big)\Big)\cr\\[-2mm]
&\leq&k\rho\big(b(f^nz-f^nx^*)\big)+k\rho\big(b(f^{n-1}x^*-f^{n-1}z)\big)\cr\\[-2mm]
&=&k\rho\big(b(f^nz-x^*)\big)+k\rho\big(b(f^{n-1}z-x^*)\big)
\end{eqnarray*}
for all $n\geq1$, which yields
$$\rho\big(b(f^nz-x^*)\big)\leq\lambda\rho\big(b(f^{n-1}z-x^*)\big),$$
where $\lambda=\frac k{1-k}\in(0,1)$ because $k<\frac12$. So by the mathematical induction, we get
$$\rho\big(b(f^nz-x^*)\big)\leq\lambda^n\rho\big(b(z-x^*)\big)\qquad n=0,1,\ldots\,.$$
Since $\lambda<1$, it follows that $bf^nz\stackrel{\rho}\longrightarrow bx^*$. Similarly, one can show that $bf^nz\stackrel{\rho}\longrightarrow by^*$, and so $bx^*=by^*$ because the limit of a $\rho$-convergent sequence in a modular space is unique. Thus, from $b>0$, it follows that $x^*=y^*$.\par
Consequently, the fixed point of $f$ is unique.
\end{proof}

Setting $G=G_0$ and $G=G_1$ once again, we get the following consequences of Theorem \ref{kannan} in modular and partially ordered modular spaces, respectively.

\begin{cor}
Let $X$ be a $\rho$-complete modular space and a mapping $f:X\to X$ satisfies
$$\rho\big(b(fx-fy)\big)\leq k\rho\big(a_1(fx-x)\big)+l\rho\big(a_2(fy-y)\big)\qquad(x,y\in X),$$
where $k,l,a_1,a_2$ and $b$ are positive with $k+l<1$, $a_1\leq\frac{b}{2}$ and
$a_2\leq b$. Then $f$ has a unique fixed point $x^*\in X$ and $bf^nx\stackrel{\rho}\longrightarrow bx^*$ for all $x\in X$.
\end{cor}

\begin{cor}
Let $\preceq$ be a partial order on a $\rho$-complete modular space $X$ such that the triple $(X,\rho,\preceq)$ has Property $(\ast\ast)$. Assume that a nondecreasing mapping $f:X\to X$ satisfies
$$\rho\big(b(fx-fy)\big)\leq k\rho\big(a_1(fx-x)\big)+l\rho\big(a_2(fy-y)\big)\qquad(x,y\in X,\ \text{and either}\ x\preceq y\ \text{or}\ y\preceq x),$$
where $k,l,a_1,a_2$ and $b$ are positive with $k+l<1$, $a_1\leq\frac{b}{2}$ and
$a_2\leq b$. Then $f$ has a fixed point if and only if there exists an $x\in X$ such that $T^nx$ is comparable to $T^mx$ for all $m,n\geq0$. Moreover, this fixed point is unique if $k<\frac12$ and each pair of elements of $X$ has either an upper or a lower bound.
\end{cor}

As another consequence of Theorem \ref{kannan}, we have the convex version of it as follows:

\begin{cor}\label{finalcor}
Let $(X,\rho)$ be a $\rho$-complete modular space endowed with a graph $G$, where $\rho$ is a convex modular, and the triple $(X,\rho,G)$ have Property $(\ast)$. Assume that $f:X\rightarrow X$ is a mapping which preserves the edges of $\widetilde G$ and satisfies
$$\rho\big(b(fx-fy)\big)\leq k\rho\big(a_1(fx-x)\big)+l\rho\big(a_2(fy-y)\big)\qquad\big(x,y\in X\ \text{and}\ (x,y)\in E(\widetilde G)\big),$$
where $k$, $l$, $a_1$, $a_2$ and $b$ are positive numbers with $b>4\max\{a_1,a_2,a_1k,a_2l\}$. Then $f$ has a fixed point if and only if $C_f\ne\emptyset$. Moreover, this fixed point is unique if $X$ satisfies Condition $(\star)$.
\end{cor}

\begin{proof}
Set $c=2\max\{a_1,a_2,a_1k,a_2l\}$ and choose any $a_0\in(c,\frac{b}{2}]$. Then by the hypothesis and convexity of $\rho$, we have
\begin{eqnarray*}
\rho\big(b(fx-fy)\big)&\leq&k\rho\big(a_1(fx-x)\big)+l\rho\big(a_2(fy-y)\big)\cr\\[-2mm]
&=&k\rho\Big(\frac{a_1}{a_0} a_0\big(fx-x\big)+\big(1-\frac{a_1}{a_0}\big)0\Big)+l\rho\Big(\frac{a_2}{a_0} a_0\big(fy-y\big)+\big(1-\frac{a_2}{a_0}\big)0\Big)\cr\\[-2mm]
&\leq&\frac{a_1k}{a_0}\rho\big(a_0(fx-x)\big)+\frac{a_2l}{a_0}\rho\big(a_0(fy-y)\big)
\end{eqnarray*}
for all $x,y\in X$ with $(x,y)\in E(\widetilde G)$. Since $a_0\leq\frac{b}{2}<b$, and
$\frac{a_1k}{a_0}+\frac{a_2l}{a_0}<1$, it follows that $f$ satisfies (K2) for the graph $\widetilde G$ with the constants $k$, $l$, $a_1$ and $a_2$ replaced with $\frac{a_1k}{a_0}$, $\frac{a_2l}{a_0}$, $a_0$ and $a_0$, respectively, and $b$ kept fixed. Since $f$ preserves the edges of $\widetilde G$, it follows that $f$ is a Kannan $\widetilde G$-$\rho$-contraction and the first assertion is concluded immediately from Theorem \ref{kannan}.\par
On the other hand, since $a_0>c\geq2a_1k$, it follows that $\frac{a_1k}{a_0}<\frac12$, and because $X$ satisfies Condition $(\star)$, Theorem \ref{kannan} guarantees the uniqueness of the fixed point of $f$.
\end{proof}

\end{document}